\documentclass[12pt]{amsart}
\usepackage[T2A]{fontenc}
\usepackage[utf8]{inputenc}   
\usepackage[russian, english]{babel}
\usepackage[shortlabels]{enumitem}

\usepackage{hyperref}
\usepackage{amsmath}
\usepackage{amsthm}
\usepackage{amsfonts}
\usepackage{amssymb}

\numberwithin{equation}{section}

\usepackage{bbm}
\usepackage{mathptmx} 
\usepackage[scaled=0.90]{helvet} 
\usepackage{courier} 
\usepackage[T1]{fontenc}
\usepackage{color}

\usepackage{graphicx}
\DeclareGraphicsExtensions{.eps}

\newtheorem{theorem}{Theorem}[section]
\newtheorem{utv*}{Proposition}
\newtheorem{hyp*}{Conjecture}
\newtheorem*{example*}{Example}
\newtheorem{lemma}[theorem]{Lemma}
\newtheorem{corollary}[theorem]{Corollary}

\newtheorem*{th*}{Theorem}
\newtheorem{prop}[theorem]{Proposition}
\theoremstyle{definition}
\newtheorem{zamech}[theorem]{Remark}
\newtheorem{defin}[theorem]{Definition}

\def\sli{\sum\limits}
\def\ili{\int\limits}

\def\a{\alpha}
\def\R{\mathbb{R}}

\def\ep{\varepsilon}
\def\vf{\varphi}

\def\E{\mathbb{E}}

\def\X{\mathcal{X}}
\def\H{\mathcal{H}}
\def\card{\textup{card}}

\newcommand{\diam}{\operatorname{diam}}

\newcommand{\dist}{\operatorname{dist}}

\newcommand{\mE}{\mathcal{E}}

\renewcommand{\P}{\mathbb{P}}

\def\cyr{\fontencoding{OT2}\fontfamily{wncyr}\selectfont}
\DeclareTextFontCommand{\textcyr}{\cyr}

{\end{list}}

\newcounter{vremennyj}

\begin{document}

\title[]{The Covering Radius of Randomly Distributed Points on a Manifold}

\date{\today}
\maketitle

\begin{center}
\textrm{A. Reznikov \footnote{\label{note1}The research of the authors was supported, in part, by the National Science Foundation grants DMS-1109266 and DMS-1412428}} and E. B. Saff \footnotemark[1]

\begin{small}Center for Constructive Approximation, Department of Mathematics, Vanderbilt University, Nashville, TN 37240
\vspace{.1cm}

e-mails:  aleksandr.b.reznikov@vanderbilt.edu,\,\, edward.b.saff@vanderbilt.edu \end{small}
\end{center}
\date{}

\begin{abstract} We derive fundamental asymptotic results for the expected covering radius $\rho(X_N)$ for
$N$ points that are randomly and independently distributed with respect to surface measure on a sphere as well as on a class of
smooth manifolds. For the unit sphere $\mathbb{S}^d \subset \mathbb{R}^{d+1}$, we obtain the
precise asymptotic that $\E\rho(X_N)[N/\log N]^{1/d}$ has limit $[(d+1)\upsilon_{d+1}/\upsilon_d]^{1/d}$
as $N \to \infty $, where $\upsilon_d$ is the volume of the  $d$-dimensional unit ball. This proves
a recent conjecture of Brauchart et al. as well as extends a result previously known only for the circle.
Likewise we obtain precise asymptotics for the expected covering radius of $N$ points randomly distributed
on a $d$-dimensional ball, a $d$-dimensional cube, as well as on a 3-dimensional polyhedron (where the points are independently distributed with
respect to volume measure).  More generally, we deduce upper and lower bounds for the expected covering radius of $N$ points that are randomly and independently distributed on a metric measure space, provided  the measure satisfies certain
regularity assumptions.

\end{abstract}

\vspace{4mm}


\vspace{2mm}


\vspace{2mm}

\normalsize

\section{Introduction and Notation}
The purpose of this paper is to obtain asymptotic results for the expected value of the covering
radius of $N$ points $X_N=\{x_1, x_2,\ldots, x_N\}$ that are randomly and independently distributed with respect to a given measure $\mu$
over a metric space $(\X, m)$. By the \emph{covering radius} $\rho(X_N, \mathcal{X})$ (also known as the \emph{mesh norm }or \emph{fill radius}) of the set $X_N$ with respect to $\X$, we mean the radius of the largest neighborhood centered at a point of $\X$ that contains no points of $X_N$; more precisely,
$$
\rho(X_N, \mathcal{X}):=\sup_{y\in \mathcal{X}}\inf_j m(y, x_j).
$$
Our focus is on the limiting behavior as $N \to \infty$ of the expected value $\E\rho(X_N, \mathcal{X}).$ \

The covering radius of a discrete point set is an important characteristic that arises in a variety of contexts.  For example, it plays an
  essential role in determining the accuracy of various numerical approximation schemes such as those involving radial basis techniques (see, e.g. \cite{FW}, \cite{MNPW}). Another area where the covering radius arises is in ``1-bit sensing'', i.e., the problem of approximating an unknown vector (signal) $x\in K$ from knowledge of $m$ numbers $\textup{sign}\langle x, \theta_j \rangle$, $j=1,\ldots, m$, where the vectors $\theta_j$ are selected independently and randomly on a sphere; see discussion after Corollary \ref{nets} for details. 

With regard to asymptotics for the expected value of the covering radius, of particular interest is the case where $\X$ is the unit sphere $\mathbb{S}^d$ in $\mathbb{R}^{d+1}$ and the metric is Euclidean
distance in $\mathbb{R}^{d+1}$. In \cite{BSR}, Bourgain, Sarnak and Rudnick study local statistics of certain spherical point configurations derived from normalized
 sums of squares of integers. Their investigation focuses on whether such configurations exhibit features of randomness,
 and for this purpose they study various local statistics, including  the covering radius of random points on $\mathbb{S}^d$. They prove
that this radius is bounded from above by $N^{-1/d +o(1)}$ as $N \to \infty.$\

 For $d=1$, i.e. the unit
circle, it is shown in \cite {DN} by using order statistics, that for $N$ points independently and randomly distributed
with respect to arclength on the circle,
$$ \lim_{N \to \infty}\mathbb{E}\rho(X_N,\mathbb{S}^1)\left[ \frac{N}{\log N}\right] =\pi.
$$
Up to now, there has been no extension of this result to higher-dimensional spheres where the order statistics approach is more elusive.
Based on a heuristic argument and numerical experiments, Brauchart et al. \cite{BHS} have conjectured that the appropriate extension of the circle case
is the following:
\begin{equation}\label{sphereequiv1}
\lim_{N\to\infty} \E\rho(X_N, \mathbb{S}^d)\cdot \left[\frac{N}{\log N }\right]^{1/d}=\left(\frac{(d+1)\upsilon_{d+1}}{\upsilon_d}\right)^{1/d}=\left(2\sqrt{\pi}\frac{\Gamma(\frac{d+2}2)}{\Gamma(\frac{d+1}2)}\right)^{1/d}	,%
\end{equation}
where $\upsilon_d:=\frac{\pi^{d/2}}{\Gamma(1+d/2)}$ is the volume of a $d$-dimensional unit ball in $\R^d$, and the points of $X_N$ are
randomly and independently distributed with respect to surface measure on $\mathbb{S}^d$ (more precisely, $d$-dimensional Hausdorff measure
$\H_d).$  Their conjecture is also consistent with a result of H. Maehara \cite{M} who obtained probabilistic estimates for the size of random
caps that cover the sphere $\mathbb{S}^2$. He showed that with asymptotic probability one, random caps with radii that are a constant factor larger than the expected radii will cover the sphere,
   whereas this asymptotic probability becomes zero when the random caps all have radii that are a factor smaller. However, his results fall short of providing a sharp asymptotic for the expected covering radius (in addition, his methods do not readily generalize to other smooth manifolds). As discussed in Section 3, our results for the sphere cannot be directly derived from Maehara's; however, his results are a direct consequence of our Corollary \ref{corsphere}.

The  main goal of this article is to provide a proof of \eqref{sphereequiv1} and its various generalizations.

We remark that for any compact metric space $(\X, m)$ with $\X$ having finite $d$-dimensional Hausdorff measure, there exists a positive constant
$C$ such that
for any $Y_N=\{y_1,\ldots, y_N\} \subset \X$, there holds
\begin{equation}\label{lowerbound}
\rho_N:=\rho(Y_N,\X) \geq \frac{C}{N^{1/d}},\quad N \geq 1.
\end{equation}
Indeed, a lemma of Frostman (see, e.g. Theorem 8.17 in \cite{Mat}) implies the existence of a finite positive measure $\mu$ on $\X$ for which
$\mu(B(x,r)) \leq (2r)^d$ for all $x \in \X$ and all $0<r\leq \diam(\X),$ where $B(x,r)$ denotes the ball centered at $x$ having
radius $r$. Consequently,
$$0<\mu(\X) \leq \sum_{i=1}^N \mu(B(y_i,\rho_N) )\leq N(2\rho_N)^d,
$$
which verifies \eqref{lowerbound}. Thus, as also remarked in \cite{BSR} and made more explicit by \eqref{sphereequiv1}, randomly distributed
points have relatively good covering properties, differing from the optimal by a factor of $(\log N)^{1/d}$.\

The outline of this paper is as follows. In Section 2, we state our probabilistic and expected covering radius estimates for general compact metric spaces, where the points
are randomly distributed with respect to a measure satisfying certain regularity conditions. Results for compact subsets of
Euclidean space are given in Section 3, including sharp asymptotic results for randomly distributed points with respect
to Hausdorff measure on rectifiable curves, smooth surfaces, bodies with smooth boundaries, $d$-dimensional cubes, and $3$-dimensional polyhedra.
The proofs of our stated results are provided in Section 5 utilizing properties established in Section 4 for a commonly arising probability function.\

We conclude this section with a listing of some notational conventions and terminology that will be utilized throughout the paper.

\begin{itemize}
\item We denote by $B(x,r)$ a closed ball in the metric space $(\X, m);$ more precisely,
 $B(x,r):=\{y\in \mathcal{X}\colon m(y,x)\leqslant r\}$. For $d$-dimensional balls in Euclidean space we write $B_d(x,r).$

\item For a positive finite Borel measure $\mu$ supported on a set $\X$, we say that a point $x$ is {\it randomly distributed over $\X$ with respect to $\mu$}, if it is distributed with respect to the probability measure $\mu/\mu(\X)$; i.e., for any Borel set $K$ it holds that $\P(x\in K) = \mu(K)/\mu(\X)$.
\item For a positive integer $s\leqslant d$, we denote by $\mathcal{H}_s$ the $s$-dimensional Hausdorff measure on the Euclidean space $\R^d$ with the Euclidean metric, normalized by $\mathcal{H}_s([0,1]^s)=1$. Thus, $\H_s(E)=\frac{\pi^{s/2}}{2^s\Gamma(1+s/2)}\H^s(E)$, where $\H^s$ is the Hausdorff measure defined in \cite{Fal}.
\item If $K$ is a subset of the Euclidean space $\R^d$, we always equip it with the Euclidean metric $m(x,y)=|x-y|.$

\item The symbols $c_1, c_2, \ldots$, and $C_1, C_2, \ldots$ shall denote positive constants that may differ from one inequality to another. These constants never depend on $N$.
\end{itemize}
\section{Main Theorems for Metric Spaces}\label{sectionmetric}


Throughout this section, we assume that $(\X, m)$ is a metric space, $\mu$ is a finite positive Borel measure supported on $\X$, and $X_N=\{x_1, \ldots, x_N\}$ is a set of $N$ points, independently and randomly distributed over $\X$ with respect to $\mu$. Our theorems provide estimates for the probability and expected values of the covering radius $\rho(X_N, \X)$ when the measure $\mu$ satisfies certain regularity conditions described by a function $\Phi$.

\begin{theorem}\label{metricabove}
Suppose $\Phi$ is a continuous non-negative strictly increasing function on $(0, \infty)$ satisfying $\Phi(r)\to 0$ as $r\to 0^+$. If there exists a positive number $r_0$ such that $\mu(B(x,r))\geqslant \Phi(r)$ holds for all $x\in \X$ and every $r<r_0$,  
then there exist positive constants $c_1$, $c_2$, $c_3$, and $\a_0$ such that for any $\a>\a_0$ we have
\begin{equation}\label{metricaboveprob}
\P\left[\rho(X_N, \X)\geqslant c_1\Phi^{-1}\left(\frac{\a\log N}{N}\right)\right]\leqslant c_2N^{1-c_3\a}.
\end{equation}
If, in addition, $\Phi$ satisfies $\Phi(r)\leqslant r^\sigma$ for all small $r$ and some positive number $\sigma$, then there exist positive constants $c_1, c_2$ such that
\begin{equation}\label{metricaboveexpect}
\E\rho(X_N, \X) \leqslant c_1 \Phi^{-1}\left(c_2 \frac{\log N }{N}\right).
\end{equation}

\end{theorem}
A lower bound for the expected covering radius is given in our next result.
\begin{theorem}\label{metricbelow}
Let $\Phi$ be a continuous non-negative strictly increasing function on $(0, \infty)$ satisfying $\Phi(r)\to 0$ as $r\to 0^+$ and the strict doubling property; i.e., for some constants $C_1, C_2>1$ and any small $r$ it holds that $C_1\Phi(r)\leqslant\Phi(2r)\leqslant C_2\Phi(r)$.
If there exists a subset $\X_1\subset \X$ with the following two properties:
\begin{enumerate}[label={\upshape(\roman*)}]
\item $\mu(\X_1)>0$;
\item there exist positive numbers $r_0$ and $c$ such that for any $x\in \X_1$ and every $r<r_0$ the regularity condition $c\Phi(r)\leqslant \mu(B(x,r))\leqslant \Phi(r)$ holds,
\end{enumerate}
then there exist positive constants $c_1$ , $c_2$, and $c_3$ such that
\begin{equation}\label{metricbelowprob}
\P\left[\rho(X_N, \X)\geqslant c_1\Phi^{-1}\left(\frac{c_2\log N-c_3\log\log N}N\right)\right]=1-o(1), \; \; N\to \infty.
\end{equation}

Consequently, there exist positive constants $c_1$ and $c_2$ such that
\begin{equation}\label{metricbelowexpect}
\E\rho(X_N, \X) \geqslant c_1 \Phi^{-1}\left(c_2 \frac{\log N}{N}\right).
\end{equation}
\end{theorem}
Combining Theorems \ref{metricabove} and \ref{metricbelow} we deduce the following.
\begin{corollary}\label{twosided}
Assume the function $\Phi$ is continuous non-negative, strictly increasing, strictly doubling, and that there exist positive numbers $r_0$ and $\sigma$ such that $\Phi(r)\leqslant r^\sigma$ for every $r<r_0$. If for some positive constants $c, C$, any $x\in \X$ and every $r<r_0$ we have
\begin{equation}\label{regularity}
c\Phi(r)\leqslant \mu(B(x, r))\leqslant C\Phi(r),
\end{equation}
then there exist positive constants $c_1, c_2, c_3, c_4$ such that for any $\ep>0$ there is a number $N(\ep)$ such that for any $N>N(\ep)$ we have
\begin{equation}\label{distribineq}
\P\left[c_1\Phi^{-1}\left(c_2\frac{\log N}{N}\right ) \leqslant \rho(X_N, \X)\leqslant c_3\Phi^{-1}\left(c_4\frac{\log N}{N}\right)\right] > 1-\ep.
\end{equation}
Moreover, there exist positive constants $C_1, C_2, C_3, C_4$ such that
\begin{equation}\label{expectineq}
C_1\Phi^{-1}\left(C_2\frac{\log N}{N}\right) \leqslant \E\rho(X_N, \X)\leqslant C_3\Phi^{-1}\left(C_4\frac{\log N}{N}\right).
\end{equation}
\end{corollary}

For recent estimates similar to \eqref{distribineq} and \eqref{expectineq} for the spherical cap discrepancy of random points on the unit sphere $\mathbb{S}^2\subset \R^3$, see Theorems $9$ and $10$ in \cite{ABD}.

An important class of sets in $\R^d$ to which Corollary \ref{twosided} applies are described in the following definition.
\begin{defin}
We call a set $\X\subset \R^d$ {\it $s$-regular} if the condition \eqref{regularity} holds for $\mu=\H_s$ and $\Phi(r)=r^s$; i.e., for some positive constants $r_0, c$, and $C$ there holds
\begin{equation}\label{sregset}
cr^s\leqslant \H_s(B_d(x,r)\cap \X)\leqslant Cr^s \; \; \mbox{for any $x\in \X$ and every $r<r_0$}.
\end{equation}
\end{defin}
\begin{zamech}
Examples of sets in Euclidean space for which Corollary \ref{twosided} holds include a cube $[0,1]^d$, a rectifiable curve $\Gamma\subset \R^d$, the unit sphere $\mathbb{S}^{d-1}\subset \R^d$, or any $s$-regular set $\X\subset \R^d$. Furthermore, the results of the Corollary \ref{twosided} hold not only for $\Phi(r)=r^s$, but for more general regularity functions, such as $\Phi(r)=r^\a \log^\beta\left(1/r\right)$, with $\a>0$ and $\beta\geqslant 0$.

In particular, Corollary \ref{twosided} applies for the ``middle $1/3$'' Cantor set $\mathcal{C}$ in $[0,1]$ with $\textup{d}\mu=\mathbbm{1}_{\mathcal{C}}\textup{d}\H_{\log2/\log3}$. We remark that for $\mu$-a.e. point $x\in \mathcal{C}$ we have
$$
\liminf_{r\to 0^+}\frac{\mu(B_1(x,r)\cap \mathcal{C})}{r^{\log2/\log3}}\not=\limsup_{r\to 0^+}\frac{\mu(B_1(x,r)\cap \mathcal{C})}{r^{\log2/\log3}};
$$
i.e., at $\mu$-a.e. point $x$ of $\mathcal{C}$ the density of $\mu$ at $x$ does not exist, which essentially precludes obtaining
 a sharp asymptotic for $\E\rho(X_N, \mathcal{C})$ (compare with \eqref{uniform} below).  However, Corollary \ref{twosided} provides the two-sided estimate
$$
c_1 \left(\frac{\log N}{N}\right)^{\log 3/\log 2} \leqslant \E\rho(X_N, \mathcal{C}) \leqslant c_2 \left(\frac{\log N}{N}\right)^{\log 3/\log 2}.
$$
\end{zamech}
\begin{zamech}
The condition in Theorem \ref{metricabove} that $\mu(B(x,r))\geqslant \Phi(r)$ for every $x\in \X$ is essential. Indeed, if we consider the set $\X=[0,1]\cup \{2\}$ with $\mu$ Lebesgue measure, then $\mu(B_1(x,r))\geqslant r$ for $x\in \X\setminus \{2\}$. However, we have $\P\left[\rho(X_N, \X)\geqslant 1\right]=1$, and so $\E\rho(X_N, \X)\geqslant 1$. The reason that inequality \eqref{metricaboveexpect} fails in this case is that for the point $x=2$ we have $\mu(B_1(x,r))=0$ for small values of $r$. However, Theorem \ref{metricabove} does apply if $\mu=m_{[0,1]}+\alpha\delta_2$, where $m_{[0,1]}$ is Lebesgue measure on $[0,1]$, $\delta_2$ is the unit point mass at $x=2$, and $\alpha>0$. In this case we get
$$
\E\rho(X_N, \X)\leqslant {C(\alpha)}\cdot \frac{\log N}N.
$$
In fact, repeating the proofs from Sections \ref{abovvvve} and \ref{sectionfrombelow} (with $K_1=[0,1]$), we obtain
$$
\lim_{N\to \infty} \E\rho(X_N, \X) \cdot \frac{N}{\log N} = \frac{1+\a}{2} \; \; \mbox{for any $\a>0$}.
$$
\end{zamech}

The above results have immediate consequences for $\ep$-nets. Since different definitions of an ``$\ep$-net'' occur in the literature, the terminologies that we use are made precise in what follows.
\begin{defin}
A subset $A$ of a metric space $(\X, m)$ is called an {\it $\ep$-net} (or {\it $\ep$-covering}) if, for any point $y\in \X$, there exists a point $x\in A$ such that $m(x,y)\leqslant \ep$. Equivalently, $A$ is an $\ep$-net if $\rho(A, \X)\leqslant \ep$.
\end{defin}
\begin{defin}
A subset $A$ of a metric space $(\X, m)$ with a positive Borel measure $\mu$ is called a {\it measure $\ep$-net} if any ball $B(y,r)$ with $\mu(B(y,r))\geqslant \ep$ intersects $A$.
\end{defin}
We remind the reader that on $\mathbb{S}^d$ with $\mu$ surface area measure $\H_d$, the minimal $\ep$-net has cardinality $c\ep^{-d}$ (for the proof see, for example, Lemma $5.2$ in \cite{V}), while the minimal measure $\ep$-net has cardinality $c\ep^{-1}$.
\begin{corollary}\label{nets}
If $\Phi$ and $\mu$ are as in the first part of Theorem \ref{metricabove}, then there exists a positive constant $c_1$ such that for any number $\a$ there is a positive constant $C_\a$ for which
$$
\P\left[X_N \; \mbox{is an $\ep$-net}\right] \geqslant 1-N^{-\a}, \; \; \; \mbox{for} \; \ep=c_1\Phi^{-1}\left(C_\a\frac{\log N}N\right).
$$

Furthermore, if the function $\Phi$ is doubling, and the measure $\mu$ satisfies the condition \eqref{regularity}, then for any positive number $\a$ there exists a positive constant $C_\a$ such that
$$
\P\left[X_N \; \mbox{is a measure $\ep$-net}\right] \geqslant 1-N^{-\a}, \; \; \; \mbox{for} \; \ep=C_\a\frac{\log N}N.
$$
\end{corollary}
By way of illustration, suppose for simplicity that $\Phi(r)=Cr^d$ for some positive constant $C$ and $\ep=\left[(\log N)/N\right]^{1/d}$, which implies that $N$ is of the order $\ep^{-d}\log(1/\ep)$. Then, from the first part of Corollary \ref{nets}, if we take $C_1\ep^{-d}\log\left(1/\ep\right)$ random points, we get an $\ep$-net ($\ep$-covering) with high probability. 

The cardinality of an $\ep$-covering of a set $K\subset \mathbb{S}^d$ plays an important role in ``1-bit compressed sensing''. The estimates for the number $m$ of random vectors $\{\theta_j\}_{j=1}^m$, essential to approximate an unknown signal $x\in K$ from knowledge of $m$ ``bits'' $\textup{sign}\langle x,\theta_j\rangle$ involve finding an $\ep$-covering of the set $K$ with $\log(N(K, \ep))\leqslant C\ep^{-2}w(K)$, where $N(K, \ep)$ is the cardinality of the covering, and $w$ is the so-called ``mean width'' of $K$. As can be seen from our results, for many sets $K$ a random set of $C\ep^{-d}\log(1/\ep)$ points satisfies this condition with high probability. For further discussion, see \cite{PV1}, \cite{PV2}.
\section{Expected Covering Radii for Subsets of Euclidean Space}\label{sectioneuclid}
In some cases we are able to ``glue'' upper and lower estimates together to obtain sharp asymptotic results.
For this purpose we state the following definitions.
\begin{defin}\label{flat}
Let $s$ be a positive integer, $s\leqslant d$. Suppose $K$ is a compact $s$-dimensional set in $\R^d$ with the Euclidean metric.

We call $K$ an {\it asymptotically flat $s$-regular} set if
for any $x\in K$ it holds that
\begin{equation}\label{uniform}
r^{-s}\mathcal{H}_s(B_d(x,r)\cap K)\rightrightarrows \upsilon_s \;\; \mbox{as $r\to 0^+$},
\end{equation}
where the convergence is uniform in $x$, and $\upsilon_s$ is the volume of the $s$-dimensional unit ball $B_s(0,1)$.

We call $K$ a {\it quasi-nice $s$-regular} set if
\begin{enumerate} [(i)]
\item $K$ is countably $s$-rectifiable; i.e., $K$ is of the form $\bigcup_{j=1}^\infty f_j(E_j) \cup G$, where $\H_s(G)=0$ and where each $f_j$ is a Lipschitz function from a bounded subset $E_j$ of $\R^s$ to $\R^d$;
\item There exist positive numbers $c, C, r_0$ such that for any $x\in K$ and any $r<r_0$ the $s$-regularity condition holds: $c r^s\leqslant \mathcal{H}_s(B_d(x,r)\cap K)\leqslant C r^s$;
\item There is a finite set $T\subset K$ such that for any $r<r_0$ and $y\in K\setminus \bigcup_{x_t\in T}B_d(x_t, r)$ it holds that $\mathcal{H}_s(B_d(y,r)\cap K)\geqslant \upsilon_s r^s$.
\end{enumerate}
\end{defin}
We remark that the appearance of the constant $\upsilon_s$ in the above definitions is quite natural. Indeed, if $K$ is a countably $s$-rectifiable compact set and $0<\H_s(K)<\infty$, then for $\mathcal{H}_s$-almost every point $x\in K$ the following holds: $r^{-s}\mathcal{H}_s(B_d(x,r)\cap K)\to \upsilon_s$ as $r\to 0^+$. For the details see the Theorem 17.6 in \cite{Mat} or Theorem 3.33 in \cite{Fal}. Thus, if any uniform limit in \eqref{uniform} exists, then it must equal $\upsilon_s$.

For asymptotically flat $s$-regular and quasi-nice $s$-regular sets we deduce the following precise asymptotics for the expected covering radius as well as its moments.

\begin{theorem}\label{manifolds}
Suppose $K\subset \R^d$ is an asymptotically flat $s$-regular or a quasi-nice $s$-regular set for integer $s\leqslant d$. Then for $X_N=\{x_1, \ldots, x_N\}$ a set of $N$ independently and randomly distributed points over $K$ with respect to the measure $\textup{d}\mu:=\mathbbm{1}_K\cdot \textup{d}\mathcal{H}_s / \mathcal{H}_s(K)$, and any $p \geq 1,$
\begin{equation}\label{asymp}
\lim_{N\to \infty} \E[\rho(X_N, K)^p]\cdot \left[\frac{N}{\log N}\right]^{p/s}=\left(\frac{\mathcal{H}_s(K)}{\upsilon_s}\right)^{p/s}.
\end{equation}
\end{theorem}
Important examples of asymptotically flat $s$-regular sets are given in the following result, which includes the verification of the conjecture of Brauchart et al. in \cite{BDSSWW} for the expected covering radius of randomly distributed points on the unit sphere.  
\begin{corollary}\label{corsphere}
Suppose $K$ is a closed $C^{(1,1)}$ $s$-dimensional embedded submanifold of $\R^d$; i.e., $0<\mathcal{H}_s(K)<\infty$ and, for any embedding $\vf$, all its first partial derivatives exist and are uniformly Lipschitz. Then $K$ is an asymptotically flat $s$-regular manifold, and thus for $N$ points independently and randomly distributed over $K$ with respect to $\textup{d}\mu=\mathbbm{1}_K \cdot \textup{d}\H_s/\H_s(K)$, equation \eqref{asymp} holds.

In particular, if $K=\mathbb{S}^d$ is a unit sphere in $\R^{d+1}$ and $p \geq 1$, then
\begin{equation}\label{sphereequiv}
\lim_{N\to\infty} \E[\rho(X_N, \mathbb{S}^d)^p]\cdot \left[\frac{N}{\log N }\right]^{p/d}=\left(\frac{(d+1)\upsilon_{d+1}}{\upsilon_d}\right)^{p/d}=\left(2\sqrt{\pi}\frac{\Gamma(\frac{d+2}2)}{\Gamma(\frac{d+1}2)}\right)^{p/d}	. %
\end{equation}
Thus \eqref{sphereequiv1} holds.
\end{corollary}
As a consequence of the corollary, we shall deduce in Section 5 the result of Maehara mentioned in the Introduction.
\begin{corollary}[Maehara \cite{M}]\label{maehara}
Suppose $X_N=\{x_1, \ldots, x_N\}$ is a set of $N$ points, independently and randomly distributed over the unit sphere $\mathbb{S}^d$ with respect to $\textup{d}\mu=\mathbbm{1}_{\mathbb{S}^d} \cdot \textup{d}\H_d/\mathcal{H}_s(\mathbb{S}^d))$ and set
$$
Z_N:=\rho(X_N, \mathbb{S}^d)\cdot \left(\frac{\upsilon_d}{(d+1)\upsilon_{d+1}}\cdot \frac{N}{\log N}\right)^{1/d}.
$$
Then $Z_N$ converges in probability to 1 as $N \to \infty$; i.e., for each $\epsilon > 0,$
\begin{equation}\label{Mae}
\lim_{N \to \infty} \mathbb{P}(|Z_N-1|\geq\epsilon)=0.
\end{equation}
\end{corollary}
\begin{zamech}
We remark that our results for $\mathbb{S}^d$ do not directly follow from \eqref{Mae}. 
Maehara's result implies that the bounded sequence
$$
p_N(t):=\P(Z_N\geqslant t)\to \mathbbm{1}_{[0,1]}(t) \; \; \mbox{for a.e. $t>0$};
$$
however, since the range of $t$ is $[0, \infty)$, the constant function $1$ is not integrable, and we cannot apply the Lebesgue dominated convergence theorem to get $\E Z_N = \int_0^{\infty} p_N(t)dt \to 1$.


\end{zamech}

The next corollary gives an example of a quasi-nice $1$-regular set.
\begin{corollary}\label{corcurve}
Suppose $\gamma$ is a rectifiable curve in $\R^d$(i.e., $0<\mathcal{H}_1(\gamma)<\infty$ and $\gamma$ is a continuous injection of a closed interval of $\R$). If $X_N$ denotes a set of $N$ points independently and randomly distributed over $\gamma$ with respect to $\textup{d}\mu:=\mathbbm{1}_\gamma\cdot \textup{d}\mathcal{H}_1/\mathcal{H}_1(\gamma)$, then $\gamma$ is a quasi-nice $1$-regular set, and for any $p\geqslant 1$
\begin{equation}\label{curveequiv}
\lim_{N\to \infty} \E[\rho(X_N, \gamma)^p]\cdot \left[\frac{N}{\log N}\right]^p=\left(\frac{\mathcal{H}_1(\gamma)}{2}\right)^p.
\end{equation}
\end{corollary}

Next we deal with the following problem: suppose $A\subset \R^d$ is a $d$-dimensional set, but the condition
$$
\mathcal{H}_d(A\cap B_d(x,r))\geqslant \upsilon_d r^{d}
$$
fails for a certain number of points $x\in A$ and the limit \eqref{uniform} in the Definition \ref{flat} is not uniform. Such situations arise for sets with boundary, which include the unit ball $B_d(0,1)$ and the unit cube $[0,1]^d$. The case of the ball is included in the next theorem, while the case of the cube is studied in the Theorem \ref{cube}.
\begin{theorem}\label{theoremball}
Let $d\geqslant 2$ and $K\subset \R^d$ a set that satisfies the following conditions.
\begin{enumerate}[label={\upshape(\roman*)}]
\item $K$ is compact and $0<\H_d(K)<\infty$;
\item $K=\textup{clos}(K_0)$, where $K_0$ is an open set in $\R^d$ with $\partial K_0 = \partial K$;
\item The boundary $\partial K$ of $K$ is a $C^2$ smooth $(d-1)$-dimensional embedded submanifold of $\R^d$.
\end{enumerate}
Let $X_N=\{x_1, \ldots, x_N\}$ be a set of $N$ points, independently and randomly distributed over $K$ with respect to $\textup{d}\mu=\mathbbm{1}_K \cdot \textup{d}\mathcal{H}_d / \H_d(K)$.
Then for any $p\geqslant 1$
\begin{equation}\label{smoothequiv}
\lim_{N\to \infty}\E[\rho(X_N, K)^p]\cdot \left[\frac{N}{\log N}\right]^{p/d}=\left(\frac{2(d-1)}d\cdot \frac{\H_d(K)}{\upsilon_d}\right)^{p/d}.
\end{equation}
In particular, for the unit ball,
\begin{equation}\label{ballequiv}
\lim_{N\to \infty}\E[\rho(X_N, B_d(0,1))^p]\cdot \left[\frac{N}{\log N}\right]^{p/d}=\left(\frac{2(d-1)}d\right)^{p/d}.
\end{equation}
\end{theorem}
\begin{zamech}
We see that in the case $d=2$ we have $2(d-1)/d=1$, and so the constant on the right-hand side of \eqref{smoothequiv} coincides with the constant for smooth closed manifolds, see \eqref{asymp}. However, when $d>2$ we have $2(d-1)/d>1$; thus this constant becomes bigger than for smooth closed manifolds.
\end{zamech}


The next two propositions deal with cases when the boundary of the set is not smooth. For simplicity, we formulate them for a cube $[0,1]^d$ and a polyhedron in $\R^3$. However, the proof can be applied to other examples, such as cylinders.
\begin{prop}\label{cube}
Suppose $d\geqslant 2$ and $[0,1]^d$ is the $d$-dimensional unit cube. Let $\textup{d}\mu=\mathbbm{1}_{[0,1]^d}\cdot \textup{d}\mathcal{H}_d$. If $X_N=\{x_1, \ldots, x_N\}$ is a set of $N$ points, independently and randomly distributed over $[0,1]^d$ with respect to $\mu$, then for any $p\geqslant 1$
\begin{equation}\label{cubeequiv}
\lim_{N \to  \infty} \E[\rho(X_N, [0,1]^d)^p]\cdot \left[\frac{N}{\log N }\right]^{p/d}=\left(\frac{2^{d-1}}{d\upsilon_d}\right)^{p/d}.
\end{equation}
\end{prop}
\begin{prop}
Suppose $P$ is a polyhedron in $\R^3$ of volume $V(P)$. Let $X_N=\{x_1, \ldots, x_N\}$ be a set of $N$ points, independently and randomly distributed over $P$ with respect to $\textup{d}\mu=\mathbbm{1}_P\cdot \textup{d}\mathcal{H}_3/V(P)$. If $\theta$ is the smallest angle at which two faces of $P$ intersect, then for any $p\geqslant 1$
\begin{equation}\label{polyhequiv1}
\lim_{N \to  \infty} \E[\rho(X_N, P)^p]\cdot \left[\frac{N}{\log N}\right]^{p/3}=\left(\frac{2\pi V(P)}{3\theta \upsilon_3}\right)^{p/3} = \left(\frac{V(P)}{2\theta}\right)^{p/3}, \;\; \mbox{if $\theta\leqslant \frac{\pi}{2}$};
\end{equation}
\begin{equation}\label{polyhequiv2}
\lim_{N \to  \infty} \E[\rho(X_N, P)^p]\cdot \left[\frac{N}{\log N}\right]^{p/3} = \left(\frac{V(P)}{\pi}\right)^{p/3}, \;\; \mbox{if $\theta\geqslant \frac{\pi}{2}$}.
\end{equation}
\end{prop}
In the theorems up to now we dealt with measures $\mu$ on sets $\X$ satisfying for all $x\in \X$ the condition $cr^s\leqslant \mu(B(x,r)\cap \X)\leqslant Cr^s$ (i.e., the regularity function $\Phi$ was the same for all points of $\X$); only the values of best constants $c, C$ differed for points $x$ deep inside $\X$ from those near the boundary. We now give an example of a measure for which
the regularity function parameter $s$ depends upon the distance to the boundary.
\begin{prop}
Consider the interval $[-1,1]$ and the measure $\textup{d}\mu=\frac{dx}{\pi\sqrt{1-x^2}}$. Let $X_N=\{x_1, \ldots, x_N\}$ be a set of $N$ points, independently and randomly distributed over $[-1,1]$ with respect to $\mu$. Define
$$\hat{\rho}(X_N, [0,1]):=\sup_{y\in [1-\frac{1}{N^a}, 1]} \inf_j |y-x_j|, \;\;\;\;\; \tilde{\rho}(X_N, [0,1]):=\sup_{y\in [-1+\frac{1}{N^a}, 1-\frac{1}{N^a}]} \inf_j |y-x_j|. $$
\begin{enumerate}[label={\upshape(\roman*)}]
\item If $a=2$, then there exist positive constants $c_1$ and $c_2$ such that
\begin{equation}\label{arcsinineq1}
\frac{c_1}{N^2}\leqslant \E\hat{\rho}(X_N, [0,1]) \leqslant \frac{c_2}{N^2}.
\end{equation}
\item If $0<a<2$, then there exist positive constants $c_1$ and $c_2$ such that
\begin{equation}\label{arcsinineq2}
\frac{c_1\log N}{N^{1+\frac{a}2}}\leqslant \E\hat{\rho}(X_N, [0,1]) \leqslant \frac{c_2\log N}{N^{1+\frac{a}2}}.
\end{equation}
\item For any $a>0$ there exist positive constants $c_1$ and $c_2$ such that
\begin{equation}\label{arcsinineq3}
\frac{c_1\log N}{N}\leqslant \E\tilde{\rho}(X_N, [0,1]) \leqslant \frac{c_2\log N}{N}.
\end{equation}
\end{enumerate}
\end{prop}
Observe that if we stay away from the endpoints $\pm1$, the measure $\mu$ acts as the Lebesgue measure, and thus the order of the expectation of the covering radius is $(\log N)/N$. However, when we are close to the points $\pm 1$ (where ``close'' depends on $N$), the measure $\mu$ acts somewhat like the Hausdorff measure $\mathcal{H}_{1/2}$, and we get a different order for the covering radius.
\section{An auxiliary function}\label{auxproofs}
The proofs of the results stated in Sections \ref{sectionmetric} and \ref{sectioneuclid} rely heavily on the properties of the following function.
For three positive numbers $N, n, m$, with $m$ and $N$ being integers and $m\leqslant n\leqslant N$, set
\begin{equation}\label{functionf}
f(N, n, m):=\sli_{k=1}^m (-1)^{k+1}\binom{m}{k} \left(1-\frac{k}{n}\right)^N.
\end{equation}

The useful fact about the function $f(N, n,m)$ is the following.
\begin{lemma}
Suppose $X_N=\{x_1, \ldots, x_N\}$ is a set of $N$ points independently and randomly distributed on a set $\X$ with respect to a Borel probability measure $\mu$. Let $B_1, \ldots, B_m$ be disjoint subsets of $\X$ each of $\mu$-measure $1/n$. Then
\begin{equation}\label{oneisempty}
\P\big(\exists k \colon B_k\cap X_N=\emptyset\big)=f(N, n, m).
\end{equation}
\end{lemma}
\begin{proof}
We use well-known formula that, for any $m$ events $A_1, \ldots, A_m$,
\begin{equation}\label{probunion}
\P\left(\bigcup_{k=1}^m A_j\right) = \sli_{k=1}^m (-1)^{k+1} \sli_{(j_1, \ldots, j_k)} \P(A_{j_1}\cap A_{j_2}\cap \cdots \cap A_{j_k}),
\end{equation}
where the integers $j_1, \ldots, j_k$ are distinct.

Let the event $A_i$ occur if the set $B_i$ does not intersect $X_N$. Then for any $k$-tuple $(j_1, \ldots, j_k)$ the event $A_{j_1}\cap \cdots \cap A_{j_k}$ occurs if the points $x_1, \ldots, x_N$ are in the complement of the union $B_{j_1}\cup \cdots \cup B_{j_k}$; i.e., $x_1, \ldots, x_N$ are in a set of measure $1-k/n$. We see that for any $k$-tuple the probability of this event is equal to $\left(1-k/n \right)^N$. Moreover, there are exactly $\binom{m}{k}$ such $k$-tuples. Therefore,
$$
\sli_{(j_1, \ldots, j_k)} \P(A_{j_1}\cap \cdots \cap A_{j_k}) =\binom{m}{k} \left(1-\frac{k}n\right)^N,
$$
and \eqref{oneisempty} follows from \eqref{probunion}.

\end{proof}

For the lower bounds in Theorems \ref{metricbelow} and \ref{manifolds} we will need the following estimate on the function $f(N,n,m)$.
\begin{lemma}
For any three numbers $0<m\leqslant n\leqslant N$, such that $m$ and $N$ are integers,
\begin{equation}\label{estimateoff}
f(N, n,m)\geqslant 1-\left[1-\left(1-\frac{1}n\right)^N\right]^m-\frac{N}{n^2}\cdot \frac{m(m-1)}2 \cdot \left(1-\frac1n\right)^{2(N-1)}\cdot \left[1+\left(1-\frac1n\right)^{N-1}\right]^{m-2}.
\end{equation}
\end{lemma}
\begin{proof}
Notice first that for $k\geqslant 1$ and $0\leqslant x \leqslant 1$ we have
$$
1-kx\leqslant (1-x)^k \leqslant 1-kx + \frac{k(k-1)}2 x^2.
$$
Thus, for $x=1/n$, we get
$$
\left(1-\frac1n\right)^k-\frac{k(k-1)}2\frac{1}{n^2} \leqslant 1-\frac{k}{n}\leqslant \left(1-\frac1n\right)^{k}
.
$$
Suppose $(1-\frac1n)^k\geqslant \frac{k(k-1)}2\frac{1}{n^2}$. Using the inequality
$$
a^N - (a-b)^N =b \cdot (a^{N-1}+(a-b)a^{N-2}+\cdots + (a-b)^{N-1})\leqslant N\cdot b \cdot a^{N-1}, \; \mbox{if $a>b>0$},
$$
we get
\begin{align}\label{brekekek}
\left(1-\frac kn\right)^N& \geqslant \left(\left(1-\frac1n\right)^k-\frac{k(k-1)}2\frac{1}{n^2}\right)^N  \\
& \geqslant \left(1-\frac1n\right)^{kN }- N\cdot \frac{k(k-1)}2\frac{1}{n^2} \cdot \left(1-\frac{1}n\right)^{k(N-1)}. \notag
\end{align}

Suppose now that $(1-\frac1n)^k< \frac{k(k-1)}2\frac{1}{n^2}$. Then
$$
\left(1-\frac1n\right)^{kN }- N\cdot \frac{k(k-1)}2\frac{1}{n^2} \cdot \left(1-\frac{1}n\right)^{k(N-1)} = \left(1-\frac1n\right)^{k(N-1)}\left(\left(1-\frac1n\right)^k - N\frac{k(k-1)}2\frac{1}{n^2}\right)<0,
$$
so as in inequality \eqref{brekekek} for $k\leqslant n$,
$$
\left(1-\frac kn\right)^N\geqslant \left(1-\frac1n\right)^{kN }- N\cdot \frac{k(k-1)}2\frac{1}{n^2} \cdot \left(1-\frac{1}n\right)^{k(N-1)}
$$
also holds.
Therefore,
\begin{multline*}
f(N,n,m)=\sli_{\text{$k$ odd, $k\leqslant m$}} \binom{m}{k} \left(1-\frac{k}n\right)^N - \sli_{\text{$k$ even, $k\leqslant m$}} \binom{m}{k}\left(1-\frac{k}n\right)^N \geqslant  \\
\sli_{\text{$k$ odd}} \binom{m}{k}\left[\left(1-\frac1n\right)^{kN }- N\cdot \frac{k\left(k-1\right)}2\frac{1}{n^2} \cdot \left(1-\frac{1}n\right)^{k\left(N-1\right)}\right] - \sli_{\text{$k$ even}} \binom{m}{k}\left(1-\frac1n\right)^{kN } \geqslant
\end{multline*}
\begin{equation}\label{koaks1}
\sli_{k=1}^m \left(-1\right)^{k+1}\binom{m}{k}\left(1-\frac1n\right)^{kN } - \frac{N}{n^2}\sli_{k=0}^m \binom{m}{k}\frac{k\left(k-1\right)}2 \cdot \left(1-\frac{1}n\right)^{k\left(N-1\right)}.
\end{equation}
The first sum in \eqref{koaks1} is equal to $1-(1-(1-\frac1n)^N)^m$. To calculate the second sum we notice that
$$
\frac{m(m-1)}{2}x^2 (1+x)^{m-2}=\frac{1}{2}x^2 ((1+x)^m)'' = \sli_{k=0}^m \binom{m}k \cdot \frac{k(k-1)}2 x^{k}.
$$
Thus, for $x=(1-\frac1n)^{N-1}$ we get
$$
\sli_{k=0}^m \binom{m}k \frac{k(k-1)}2 \cdot \left(1-\frac{1}n\right)^{k(N-1)} = \frac{m(m-1)}2 \left(1-\frac1n\right)^{2(N-1)}\cdot \left(1+\left(1-\frac1n\right)^{N-1}\right)^{m-2}.
$$
Combining the above estimates we obtain \eqref{estimateoff}.
\end{proof}
With the help of \eqref{estimateoff} we can deduce some asymptotic properties of $f(N,n,m)$ as $N\to \infty$.
\begin{lemma}\label{lemmaforf}
Let $N$ be a positive integer and $n,m$ be numbers satisfying $1\leqslant m\leqslant n\leqslant N$. Further, let $\kappa_n$ denote constants depending on $n$ such that $0<c_1\leqslant\kappa_n\leqslant c_2$ for all $n$.
\begin{enumerate}[label={\upshape(\roman*)}]
\item If $m=\left\lfloor \kappa_n n \right\rfloor$ and $c_2\leqslant 1$, then there exists a number $\a$ such that for $n= \frac{N}{\log N-\alpha \log\log N}$ we have $f(N,n,m)\to 1$ as $N\to \infty$.
\item If $d>1$ and $m=\left\lfloor \kappa_n n^{\frac{d-1}{d}}\right\rfloor$, then there exists a number $\alpha$ such that for $n= \frac{N}{\frac{d-1}d\log N-\a\log\log N}$ we have $f(N, n, m)\to 1$ as $N\to \infty$.
\item If $d>1$ and $m=\left\lfloor \kappa_n n^{\frac{1}{d}}\right\rfloor$, then there exists a number $\alpha$ such that for $n= \frac{N}{\frac{1}d\log N-\a\log\log N}$ we have $f(N, n, m)\to 1$ as $N\to \infty$.
\end{enumerate}
\end{lemma}
\begin{proof}
We prove only part (i) since the proofs of the second and third parts are similar.

In what follows, to simplify the displays, we omit the symbol for the integer part. If $a_N$ and $b_N$ are two sequences of positive numbers, we write $a_N\sim b_N$ to mean $a_N/b_N\to 1$ as $N\to \infty$.

For our choice of $n$ in part (i) we have
$$
\left(1-\frac{1}{n}\right)^N \sim \exp\left(-\frac{N}n\right) \sim \frac{(\log N)^\a}{N}.
$$
Thus,
$$
\left(1-\left(1-\frac{1}{n}\right)^N\right)^{\kappa_n n} \sim  \left(1-\frac{(\log N)^\a}{N}\right)^{\frac{\kappa_n N}{\log N-\a\log\log N}}\sim \exp\left(-\frac{\kappa_n (\log N)^\a}{\log N-\a\log\log N}\right).
$$
If $\a>1$, then the last expression tends to zero. Moreover,
\begin{align*}
\frac{N}{n^2}\cdot \frac{m(m-1)}2 &\left(1-\frac1n\right)^{2(N-1)}\cdot  \left(1+\left(1-\frac1n\right)^{N-1}\right)^{m-2} \\
&\sim \frac{\kappa_n^2}2\cdot\frac{(\log N)^{2\a}}{N}\cdot \left(1+\frac{(\log N)^\a}{N}\right)^{\frac{\kappa_n N}{\log N-\a\log\log N}} \\
&\sim \frac{\kappa_n^2}2\cdot\frac{(\log N)^{2\a}}{N}\cdot \exp\left(\kappa_n(\log N)^{\a-1}\right).
\end{align*}
For $\a=3/2$ (actually, any $0<\a<2$ will work) the last expression is comparable to
$$
\frac{(\log N)^3}{N}\exp(\kappa_n(\log N)^{\frac12}),
$$
which tends to zero as $N$ tends to infinity.
Thus from \eqref{estimateoff} we deduce that $\liminf_{N\to \infty} f(N,n,m) \geqslant~ 1$. However, since $f(N,n,m)$ is equal to a certain probability, we have that $f(N, n, m)\leqslant ~1$, and so $\lim_{N\to \infty}f(N,n,m)=1$.
\end{proof}

\section{Proofs}\label{proofs}
\subsection{Preliminary objects}\label{sectionprelim}
Fix a compact set $\X_0$ with a metric $m$. For any large positive number $n$ let $\mE_n(\X_0)$ be a maximal set of points such that for any $y,z\in \mE_n$ we have $m(y,z)\geqslant 1/n$. Then for any $x\in \X_0$ there exists a point $y\in \mE_n$ such that $m(x,y)\leqslant 1/n$ (otherwise we can add $x$ to $\mE_n$, which contradicts its maximality).
In what follows we will clearly indicate the set $\X_0$, and then just write $\mE_n$.
\subsection{Proof of the Theorem \ref{metricabove}}
Recall that $(\X, m)$ is a metric space, and $B(x,r)$ denotes a closed ball (in the metric $m$) with center $x\in \X$ and radius $r$. Put $\mE_N:=\mE_N(\X)$ and
note that
\begin{equation}\label{cardbound}
\mu(\X)\geqslant \sli_{x\in \mE_n}\mu\left(B\left(x, \frac{1}{3n}\right)\right) \geqslant \card(\mE_n) \Phi(1/(3n)). 
\end{equation}
Suppose now that $X_N=\{x_1, \ldots, x_N\}$ is a set of $N$ random points, independently distributed over $\X$ with respect to the measure $\mu$. We denote its covering radius by
$$
\rho(X_N):=\rho(X_N, \X).
$$
Suppose $\rho(X_N)> \frac{2}n$. Then there exists a point $y\in \X$ such that $X_N\cap B\left(y, \frac{2}n\right)=\emptyset$. Choose a point $x\in \mE_n$ such that $m(x,y)<\frac{1}n$. Then $B(x, \frac{1}{n})\subset B(y, \frac{2}n)$, and so the ball $B(x, \frac{1}{n})$ (and thus $B(x, \frac{1}{3n})$) does not intersect $X_N$. Therefore,
\begin{align}\label{aboveprobbb}
\P\left(\rho(X_N)\geqslant \frac{2}{n}\right) &\leqslant \P\left(\exists x\in \mE_n\colon B(x, 1/(3n))\cap X_N=\emptyset\right)  \\
& \leqslant \card(\mE_n)\cdot \left(1- \frac{\Phi(\frac1{3n})}{\mu(\X)}\right)^N. \notag 
\end{align}
We now choose $n$ to be such that $\frac{1}{3n}=\Phi^{-1}(\frac{\alpha \log N}{N})$. There exists such an $n$ since $\Phi$ is continuous and $\Phi(r)\to 0$ as $r\to 0^+$. Then utilizing the upper bound for $\card(\mE_n)$ from \eqref{cardbound}, we deduce that for some $C>0$ we have
$$
\P\left[\rho(X_N)\geqslant \frac{2}n\right] \leqslant C \frac{N}{\log N}\cdot N^{-C\alpha},
$$
which concludes the proof of the estimate \eqref{metricaboveprob}.

To establish the estimate \eqref{metricaboveexpect}, notice that since for small values of $r$ we have $\Phi(r)\leqslant r^\sigma$, it follows that  for small $r$ and $D=\frac1\sigma$ we have $\Phi^{-1}(r)\geqslant r^D$. Choose $\alpha$ so large that $N^{1-C\alpha}=o(N^{-D})$ as $N\to \infty$. Then
$$
\E\rho(X_N) \leqslant \frac{2}{n} + C\diam(\X)\cdot o(N^{-D}) =6\Phi^{-1}(\frac{\alpha \log N}{N})+o(N^{-D}).
$$
Finally, since $\Phi^{-1}(\frac{\alpha\log N}{N})\geqslant \Phi^{-1}(N^{-1}) \geqslant N^{-D}$, inequality \eqref{metricaboveexpect} follows. \hfill $\Box$

\subsection{Proof of the Theorem \ref{metricbelow}}\label{proofofmetricbelow}
Let $\mE_n:=\mE_n(\X_1)$, where $\X_1$ is as in the hypothesis. 
Notice that
\begin{equation}\label{cardbound2}
0<\mu(\X_1) \leqslant \sli_{x\in \mE_n} \mu\left(B(x, \frac{1}{n})\right) \leqslant \card(\mE_n)\Phi\left(\frac{1}n\right).
\end{equation}
An estimate as in \eqref{cardbound} together with the doubling property of $\Phi$ imply that
$$
\mu(\X)\geqslant c \cdot \card(\mE_n) \Phi\left(\frac1{3n}\right) \geqslant \tilde{c} \cdot \card(\mE_n)\Phi\left(\frac1n\right).
$$
Thus, $\tau_n:=\card(\mE_n)\cdot \Phi(1/n)$ satisfies $0<c_1<\tau_n<c_2$ for some constants $c_1$ and $c_2$ independent of $n$.
Clearly if a ball $B(x, \frac{1}{3n})$ does not intersect $X_N$, then $\rho(X_N)=\rho(X_N, \X)\geqslant \frac{1}{3n}$. Thus
$$
\P\left(\rho(X_N)\geqslant \frac{1}{3n}\right)\geqslant \P\left(\exists x\in \mE_n\colon B(x, 1/(3n))\cap X_N = \emptyset \right).
$$
Notice that the balls $B(x, \frac{1}{3n})$ are disjoint for $x\in \mE_n$, and their $\mu$-measure is comparable to $t:=\Phi(\frac{1}{n})$.

Next we claim that for every $x\in \mE_n$ there exists a constant $c_x\leqslant 1$ such that the balls $B(x, c_x \frac{1}{3n})$ have the same measure $c_0 \Phi(\frac{1}{n})=c_0 t$, and moreover that the uniform estimate $c_x > c>0$ holds for some constant $c$.
To see this, take two points $x_1, x_2\in \X_1$ and assume that the balls $B_i := B(x_i, r)$, $i=1,2$ are disjoint. Suppose $\mu(B_1)<\mu(B_2)$. Define the function $\varphi(s):=\mu(B(x_2, s\cdot r))$. The strict doubling property of $\Phi$ implies
$$
\mu(B_1)\geqslant c\Phi(r) \geqslant c\cdot C_1^k \Phi(r/2^k).
$$
Choose $k$ such that $c\cdot C_1^k \geqslant 1$. Then
$$
\mu(B_1)\geqslant \Phi(r/2^k) \geqslant \varphi(2^{-k}).
$$
Thus, $\varphi(2^{-k})\leqslant\mu(B_1)<\mu(B_2)=\varphi(1)$. By continuity of $\varphi$ we see that there exists a constant $c_{x_2}$ such that $\mu(B(x_2, c_{x_2}r))=\mu(B(x_1, r))$. Notice that $c_{x_2} \geqslant 2^{-k}=:c_0$, where $k$ depends only on the constants $c, C_1$ from Theorem \ref{metricbelow} and not on $x_1, x_2$, or $r$. Applying this procedure to all balls $B(x, 1/(3n))$, $x\in \mE_n$, and using the fact that $\card(\mE_n)=\tau_n /t$, we obtain
\begin{align}\label{blahblahblah}
\P\left(\rho(X_N)\geqslant \frac{c_0}3 \Phi^{-1}(t)\right)&\geqslant \P\left(\mbox{one of $\frac{\tau_n}t$ disjoint balls of measure $c_0t$ is disjoint from $X_N$}\right) \notag
\\&= f(N, \frac{1}{c_0t}, \frac{\tau_n}t) = f(N, \frac1{c_0t}, \frac{\kappa_n}{c_0t}),
\end{align}
where $\kappa_n:=c_0\tau_n$ and $f$ is given in \eqref{functionf}. If necessary, we can decrease the size of $c_0$ so that $\kappa_n \leqslant 1$ for $n$ large.
As we have seen in Lemma \ref{lemmaforf}(i), there exists a number $\a$ such that if
$$\frac1{c_0t}= \frac{N}{\log N-\a\log\log N},$$
then $f(N, \frac{1}{c_0t}, \frac{\kappa_n}{c_0t})\to 1$ as $N\to \infty$.
Thus, for any sufficiently large number $N$ we have
$$
\P\left(\rho(X_N)\geqslant \frac{c_0}3 \Phi^{-1}(\frac{\log N-\a\log\log N}{c_0 N})\right)\geqslant 1-o(1), \; N\to \infty,
$$
which is the desired inequality \eqref{metricbelowprob}.

Moreover, for large values of $N$ we have $\log N-\a\log\log N\geqslant \frac{1}2\log N$; thus
$$
\E \rho(X_N)\geqslant c_1 \Phi^{-1}(c_2 \frac{\log N}N),
$$
which proves inequality \eqref{metricbelowexpect}. \hfill $\Box$

\subsection{Estimates from above for asymptotically flat sets}\label{abovvve}
Let $K$ be an asymptotically flat $s$-regular subset of $\R^d$ and put
$$
\rho(X_N)=\rho(X_N, K), \qquad \ep_N:=\frac{1}{\log N}.
$$
In order to deduce sharp asymptotic results we first improve our estimates from above by considering a better net of points.
For each $N>4$ let $\mE_{n/\ep_N}:=\mE_{n/\ep_N}(K)$.
From estimates similar to \eqref{cardbound} and \eqref{cardbound2} we see that $\card(\mE_n)$ is comparable to $\left(n/\ep_N\right)^s$ independently of $N$.

Suppose $\rho(X_N)> \frac{1}{n}$. Then, since $K$ is compact, for some $y\in K$ we have $B_d(y,\frac{1}{n})\cap X_N=\emptyset$, and thus there exists a point $x\in \mE_{n/\ep_N}$ such that $B_d(x, \frac{1-\ep_N}{n})\cap X_N=\emptyset$. We fix a number $\delta$, $0<\delta<1$, and take $n$ so large that
$$
\mathcal{H}_s\left(B_d(x, \frac{1-\ep_N}n)\cap K\right)\geqslant (1-\delta)\upsilon_s \frac{(1-\ep_N)^s}{n^s}
\geqslant (1-\delta)\upsilon_s \frac{1-s\ep_N}{n^s}.
$$
As in \eqref{aboveprobbb},
\begin{equation}\label{probmanifoldabove}
\P\left(\rho(X_N)> \frac{1}{n}\right)\leqslant C  \left(\frac{n}{\ep_N}\right)^s \left(1-\frac{1}{\mathcal{H}_s(K)}(1-\delta)\upsilon_s\frac{1-s\ep_N}{n^s}\right)^N.
\end{equation}
Fix a number $A>0$ and choose
$$
n_1 := \left(\frac{(1-\delta)\upsilon_s}{\mathcal{H}_s(K)} \frac{N}{\log N+A\log\log N}\right)^{1/s}.
$$
Then with $n=n_1$ in \eqref{probmanifoldabove} we get for all $N$ large,
\begin{equation}\label{loglog}
\P\left(\rho(X_N)> \frac{1}{n_1}\right) \leqslant C  \cdot N (\log N )^{s-1} e^{-(1-s/\log N)(\log N+A\log\log N)}.
\end{equation}
Recall that $C$ does not depend on $N$. Thus if $A$ and $N$ are sufficiently large, it follows that
\begin{equation}\label{problessthanlog}
\P\left(\rho(X_N)> \frac{1}{n_1}\right) \leqslant \frac{1}{\log N}.
\end{equation}
Furthermore, if we plug $n=n_2 := \left(\frac{N}{B\log N}\right)^{1/s}$ in \eqref{probmanifoldabove} we get for sufficiently large $B$
\begin{equation}\label{problessthanpower}
\P\left(\rho(X_N)> \frac{1}{n_2}\right) \leqslant N^{-p/s - 1}.
\end{equation}

With $\textup{d}\mu=\mathbbm{1}_K \textup{d}\H_s/\H_s(K)$, we make use of the formula
\begin{multline}\label{superest}
\E[\rho(X_N)^p] = \int_{K^N} \rho(X_N)^p \textup{d}\mu(x_1)\ldots \textup{d}\mu(x_N) = \ili_{\rho(X_N)\leqslant1/n_1} \rho(X_N)^p\textup{d}\mu(x_1)\ldots \textup{d}\mu(x_N) + \\ \ili_{1/n_1 < \rho(X_N)\leqslant 1/n_2} \rho(X_N)^p\textup{d}\mu(x_1)\ldots \textup{d}\mu(x_N) + \ili_{\rho(X_N)>  1/n_2}\rho(X_N)^p\textup{d}\mu(x_1)\ldots \textup{d}\mu(x_N) \leqslant \\ \frac1{n_1^p} + \frac{1}{n_2^p}\cdot \P\left(\rho(X_N)> \frac{1}{n_1}\right) + (\textup{diam}(K))^p \cdot \P\left(\rho(X_N)> \frac1{n_2}\right).
\end{multline}
From \eqref{problessthanlog}, \eqref{problessthanpower}, and the definitions of $n_1$ and $n_2$, we obtain
\begin{multline}\label{superest2}
\E[\rho(X_N)^p] \leqslant \left(\frac{\log N+A\log\log N}{N}\right)^{p/s} \cdot \left(\frac{\mathcal{H}_s(K)}{\upsilon_s}\right)^{p/s} \cdot (1-\delta)^{-p/s} + C\left(\frac{\log N}{N}\right)^{p/s} \frac{1}{\log N}  \\ +CN^{-p/s-1}.
\end{multline}
Therefore, for any $\delta$ with $0<\delta<1$,
$$
\limsup_{N \to  \infty} \E[\rho(X_N)^p]\cdot \left(\frac{N}{\log N}\right)^{p/s} \leqslant (1-\delta)^{-p/s}\cdot \left(\frac{\mathcal{H}_s(K)}{\upsilon_s}\right)^{p/s},
$$
and consequently
\begin{equation}\label{flatabovelastformula}
\limsup_{N \to  \infty} \E[\rho(X_N)^p] \cdot \left(\frac{N}{\log N}\right)^{p/s} \leqslant  \left(\frac{\mathcal{H}_s(K)}{\upsilon_s}\right)^{p/s}.
\end{equation} \hfill $\Box$
\subsection{Estimate from above for quasi-nice sets}\label{abovvvve}
Let $K$ be a quasi-nice $s$-regular subset of $\R^d$, and again set $\ep_N:=1/\log N$ and $\mE_{n/\ep_N}:=\mE_{n/\ep_N}(K)$, where $n/\ep_N\to \infty$ as $N\to \infty$. Since the set $T$ from part (iii) of Definition \ref{flat} is finite, the regularity condition (ii) implies
$$
\mathcal{H}_s\left(\bigcup_{x\in T}B_d(x, r)\right) \leqslant C \cdot \card(T) \cdot r^s = C_1 r^s, \; \; 0<r<r_0.
$$
Suppose $y_1, \ldots, y_k \in \mE_{n/\ep_N} \cap \bigcup_{x\in T}B_d(x, \frac{1-\ep_N}n)$. Then the balls $B_d(y_j, \frac{\ep_N}{3n})$ are disjoint and $B_d(y_j, \frac{\ep_N}{3n})\subset \bigcup_{x\in T}B_d(x, \frac{1+\ep_N}n)$ for $j=1,\ldots,k$. The chain of inequalities
$$
C_1\left(\frac{1+\ep_N}n\right)^s \geqslant \mathcal{H}_s\left(\bigcup_{x\in T}B_d(x,  \frac{1+\ep_N}n)\right)\geqslant \sli_{j=1}^k \mathcal{H}_s\left(B_d(y_j, \frac{\ep_N}{3n})\right) \geqslant c\cdot k\cdot (\frac{\ep_N}n)^s
$$
implies that $k\leqslant C_2/\ep_N^s$, and $C_2$ does not depend on $N$. Further, if $y\in \mE_{n/\ep_N}\setminus\bigcup_{x\in T}B_d(x, \frac{1-\ep_N}n)$, then $\H_s\left(B_d(y, \frac{1-\ep_N}n)\right)\geqslant \upsilon_s \left(\frac{1-\ep_N}n\right)^s$.

As we have seen in \eqref{probmanifoldabove}, $\P(\rho(X_N)> 1/n)$ is bounded from above by the probability that for some $y\in \mE_{n/\ep_N}$ we have $B_d\left(y, \frac{1-\ep_N}n\right)\cap X_N=\emptyset$. Taking into account that $\card(\mE_{n/\ep_N}) \leqslant C_3(n/\ep_N)^s$, we obtain
\begin{multline*}
\P\left(\rho(X_N)> \frac1n\right) \leqslant \P\bigg(\mbox{one of $\leqslant\frac{C_2}{\ep_N^s}$ balls of measure $\geqslant\frac{c_1}{n^s}$ is disjoint from $X_N$ or} \\ \mbox{one of $\leqslant C_3\left(\frac{n}{\ep_N}\right)^s$ balls of measure $\geqslant\frac{\upsilon_s(1-\ep_N)^s}{n^s}$ is disjoint from $X_N$}\bigg).
\end{multline*}
This last probability is bounded from above by
$$
\frac{C_2}{\ep_N^s}\left(1-\frac{c_1}{n^s}\right)^N + C_4\left(\frac{n}{\ep_N}\right)^s\left(1-\frac{1}{\H_s(K)}\frac{\upsilon_s (1-\ep_N)^s}{n^s}\right)^N.
$$
As in the preceding proof, if
$$
n_1=\left(\frac{\upsilon_s}{\mathcal{H}_s(K)} \frac{N}{\log N+A\log\log N}\right)^{1/s},
$$
then, for $N$ large,
$$
C_4\left(\frac{n_1}{\ep_N}\right)^s\cdot\left(1-\frac{1}{\H_s(K)}\frac{\upsilon_s (1-\ep_N)^s}{n_1^s}\right)^N \leqslant \frac{C_5}{\log N}.
$$
Furthermore notice that if $C_6$ is sufficiently large, then
$$
\frac{C_2}{\ep_N^s}\left(1-\frac{c_1}{n_1^s}\right)^N \leqslant C_6(\log N)^s N^{-c_2}, \;\; N\to \infty.
$$
Repeating estimates \eqref{superest} and \eqref{superest2}, we obtain
\begin{equation}\label{lastinquasinice}
\limsup_{N \to  \infty} \E[\rho(X_N)^p] \cdot \left(\frac{N}{\log N}\right)^{p/s} \leqslant  \left(\frac{\mathcal{H}_s(K)}{\upsilon_s}\right)^{p/s}.
\end{equation} \hfill $\Box$

Note that \eqref{lastinquasinice} holds whether or not $K$ is countably $s$-rectifiable; it requires only that properties (ii) and (iii) of Definition \ref{flat} hold.

\subsection{Estimate from below for quasi-nice sets}\label{sectionfrombelow}
For the proof of Theorem \ref{manifolds}, it remains in view of inequalities \eqref{flatabovelastformula} and \eqref{lastinquasinice}, to establish
\begin{equation}\label{liminfbigger}
\liminf_{N \to  \infty} \E[\rho(X_N)^p] \cdot \left(\frac{N}{\log N}\right)^{p/s} \geqslant  \left(\frac{\mathcal{H}_s(K)}{\upsilon_s}\right)^{p/s}
\end{equation}
for asymptotically flat and quasi-nice $s$-dimentional manifolds $K$. Since by the H\"older inequality we have
$$
\liminf_{N\to \infty}\E[\rho(X_N)^p] \cdot \left[\frac{N}{\log N}\right]^{p/s} \geqslant \left(\liminf_{N\to \infty}\E \rho(N_N)\cdot \left[\frac{N}{\log N}\right]^{1/s}\right)^p,
$$
it is enough to prove \eqref{liminfbigger} for $p=1$.
If $K$ is quasi-nice, then $K$ is countably $s$-rectifiable ($s$ is an integer) and $0<\mathcal{H}_s(K)<\infty$; thus as previously remarked, the following holds for $\H_s$-almost every point $x\in K$:
$$
r^{-s}\cdot \mathcal{H}_s(B_d(x,r)\cap K)\to \upsilon_s, \;\; r\to 0^+.
$$
Fix a number $\delta$ with $0<\delta<1$ and define $r_n:=1/n$ and $q_n:=\left(\frac{1-\delta}{1+\delta}\right)^{1/s}\cdot 1/n$, where $\{n\}$ is a given countable sequence tending to infinity.
By Egoroff's theorem, there exists a set $K_1=K_1(\delta)\subset K$ with $\mathcal{H}_s(K_1)>\frac{1}{2}\mathcal{H}_s(K)$ on which the above limit is uniform for radii $r$ equal to $r_n$ and $q_n$. That is,
\begin{equation}\label{uniforr}
r^{-s}\mathcal{H}_s(B_d(x,r)\cap K)\rightrightarrows \upsilon_s, \;\; r=r_n \; \mbox{or} \; r=q_n, \;\; n\to \infty.
\end{equation}
This means that there exists a large number $n(\delta)$, such that for any $n>n(\delta)$ we have, for every $x\in K_1$,
\begin{align}
&(1-\delta)\upsilon_s r_n^s \leqslant \mathcal{H}_s(B_d(x,r_n)\cap K) \leqslant (1+\delta)\upsilon_sr_n^s, \label{tilitili1} \\
&(1-\delta)\upsilon_sq_n^s \leqslant \mathcal{H}_s(B_d(x,q_n)\cap K) \leqslant (1+\delta)\upsilon_s q_n^s = (1-\delta)\upsilon_s r_n^s. \label{tilitili2}
\end{align}

Recalling the notation of Section \ref{sectionprelim}, we set $\mE_{n/2}:=\mE_{n/2}(K_1)$. Then, as in the preceding sections, there exist positive constants $c_1$ and $c_2$ (independent of $n$) such that $c_1 n^s \leqslant \card(\mE_{n/2}) \leqslant c_2 n^s$ where, for the lower bound, we use
$$
0<\H_s(K_1)\leqslant \H_s\left(\bigcup_{x\in \mE_{n/2}} (B_d(x, 2/n)\cap K)\right) \leqslant C\cdot \card(\mE_{n/2})(2/n)^s.
$$ Thus, $\tau_n:=\card(\mE_{n/2})/n^s$ satisfies $0<c_1\leqslant \tau_n \leqslant c_2$. Clearly, if for some $x\in \mE_{n/2}$ the ball $B_d(x, \frac1n)$ is disjoint from $X_N$, then $\rho(X_N)\geqslant \frac 1n$.
Thus, for a given $\delta>0$ and sufficiently large $n$ we have a family $\{B_d(x, 1/n)\cap K \colon x\in \mE_{n/2}(K_1)\}$ of $\tau_n n^s$ balls (relative to $K$) with disjoint interiors of radius $1/n$ and $\H_s$-measure between $(1-\delta)\upsilon_s/n^s$ and $(1+\delta)\upsilon_s /n^s$.
For a fixed $x\in \mE_{n/2}(K_1)$, define $\varphi(s):=\H_s(B(x, s/n)\cap K)$. Then $\varphi(1)\geqslant (1-\delta)\upsilon_s/n^s$. On the other hand, inequalities \eqref{tilitili2} imply
$$
\varphi\left(\left(\frac{1-\delta}{1+\delta}\right)^{1/s}\right) \leqslant (1-\delta)\upsilon_s/n^s.
$$

\noindent Thus, there is a number $c_x=c_{x,n}$, with $c_x\geqslant (\frac{1-\delta}{1+\delta})^{1/s}$, such that $\varphi(c_x)=(1-\delta) \upsilon_s/n^s$.
That is, there exists a new family $\{B_d(x, c_x/n)\cap K\colon x\in \mE_{n/2}(K_1)\}$, with $c_x \geqslant (\frac{1-\delta}{1+\delta})^{1/s}$, and the sets $B_d(x, c_x/n)\cap K$ all have the same $\H_s$ measure, namely $(1-\delta)\upsilon_s/n^s$.

As in \eqref{blahblahblah}, it follows that
\begin{align}\label{onemoreestimate}
\P\left(\rho(X_N)\geqslant \left(\frac{1-\delta}{1+\delta}\right)^{1/s}\frac1{n}\right) &\geqslant f\left(N, \frac{\mathcal{H}_s(K)n^s}{(1-\delta)\upsilon_s}, \tau_n n^s\right)\\ &= f\left(N, \frac{\mathcal{H}_s(K)n^s}{(1-\delta)\upsilon_s}, \kappa_n \cdot\frac{\mathcal{H}_s(K)n^s}{(1-\delta)\upsilon_s}\right),\notag
\end{align}
where
$$
\kappa_n:=\tau_n \cdot \frac{(1-\delta)\upsilon_s}{\mathcal{H}_s(K)}.
$$
It is easily seen that
$$
\H_s(K)\geqslant \tau_n n^s \cdot \frac{(1-\delta) \upsilon_s}{n^s} = \H_s(K)\kappa_n;
$$
thus $\kappa_n \leqslant 1$.
Part (i) of Lemma \ref{lemmaforf} therefore implies that the sequence in \eqref{onemoreestimate} tends to $1$ as $N\to \infty$ if (for suitable $\a$) we have
$$
\frac{(1-\delta) \upsilon_s}{\mathcal{H}_s(K)n^s} = \frac{\log N-\a\log\log N}{N},
$$
which is equivalent to
\begin{equation}\label{newdefinofn}
n := \left[\frac{(1-\delta) \upsilon_s}{\mathcal{H}_s(K)}\cdot \frac{N}{\log N-\a\log\log N}\right]^{1/s}.
\end{equation}
We take $N$ so large that $n$ exceeds $n(\delta)$, which ensures that the inequalities \eqref{tilitili1}--\eqref{tilitili2} hold.
From \eqref{onemoreestimate} we obtain
$$
\E\rho(X_N)\geqslant \left(\frac{1-\delta}{1+\delta}\right)^{1/s}\frac{1}{n} \cdot f\left(N, \frac{\mathcal{H}_s(K)n^s}{(1-\delta)\upsilon_s}, \tau_n n^s\right).
$$
Using the definition of $n$ in \eqref{newdefinofn}, we get
\begin{multline}\label{manifbelowlastformula}
\E \rho(X_N) \cdot  \left[\frac{N}{\log N}\right]^{1/s} \geqslant \\ \left[\frac{N}{\log N}\right]^{1/s}\cdot \left(\frac{1-\delta}{1+\delta}\right)^{1/s} \cdot \left[\frac{\mathcal{H}_s(K)}{(1-\delta) \upsilon_s}\cdot \frac{\log N-\a\log\log N}{N}\right]^{1/s}  \cdot f\left(N, \frac{\mathcal{H}_s(K)n^s}{(1-\delta)\upsilon_s}, \tau_n n^s\right),
\end{multline}
and passing to the $\liminf$ as $N\to \infty$ yields
$$
\liminf_{N \to  \infty} \E\rho(X_N) \cdot \left(\frac{N}{\log N}\right)^{1/s} \geqslant \left(\frac{1}{1+\delta}\right)^{1/s} \cdot \left[\frac{\mathcal{H}_s(K)}{\upsilon_s}\right]^{1/s}.
$$
Recalling that $\delta$ can be taken arbitrarily small, we obtain \eqref{liminfbigger} for quasi-nice sets. For asymptotically flat sets the same (but even simpler) argument applies.
\hfill $\Box$
\subsection{Proof of Corollary \ref{maehara}}
Recall that
$$
Z_N=\rho(X_N, \mathbb{S}^d)\cdot \left(\frac{\upsilon_d}{(d+1)\upsilon_{d+1}}\cdot \frac{N}{\log N}\right)^{1/d}.
$$
Corollary \ref{corsphere} implies that $\E Z_N \to 1$ and $\E[Z_N^2]\to 1$; thus $\E[(Z_N-1)^2] = \E[Z_N^2]-2\E Z_N + 1 \to 0$. The Chebyshev inequality then implies
$$
\P(|Z_N-1|>\ep) \leqslant \frac{\E[(Z_N-1)^2]}{\ep^2} \to 0,
$$
which completes the proof.
\hfill $\Box$

\subsection{Proof of the Corollaries \ref{corsphere} and \ref{corcurve}}
It is well known that a closed $C^{(1,1)}$ manifold is an asymptotically flat set, and a rectifiable curve is a quasi-nice $1$-dimensional set.
For the first fact, we refer the reader to a textbook on Riemannian geometry, for instance, \cite[Chapters 5--10]{BurIv}. The second fact can be deduced from \cite[Section 3.2]{Fal}.
\subsection{Proof of the Theorem \ref{theoremball}: estimate from above}
The proof of the theorem is similar to the proof for asymptotically flat sets. However, we need to take into account that the limit \eqref{uniform} is not equal to $\upsilon_d$ for points on the boundary. We use properties (ii) and (iii) of $K$ to obtain
\begin{align}
&r^{-d}\H_d(B_d(x,r)\cap K) \rightrightarrows \frac12 \upsilon_d, \; \; r\to 0, \;\; x\in \partial K;  \label{fact1}\\
& x\in K, \; \textup{dist}(x, \partial K)>r \Rightarrow \H_d(B_d(x,r)\cap K)=\H_d(B_d(x,r))=\upsilon_d r^d;  \label{fact2} \\
& \forall \delta>0 \; \exists r(\delta)>0 \colon \forall r<r(\delta), \forall x\in K\colon \H_d(B_d(x,r)\cap K)\geqslant (\frac12-\delta)\upsilon_d r^d. \label{fact3}
\end{align}
For the details, we refer the reader to Lee, \cite[Chapter 5]{L}
For large $N$, set $\mE_{n/\ep_N}:=\mE_{n/\ep_N}(K)$ and $\ep_N:=1/\log N$, where $n(N)$ is a sequence such that $n\asymp (N/\log N)^{1/d}$. We now fix a number $\delta$ with $0<\delta<1/2$.
Notice that if $x\in \mE_{n/\ep_N}$ and $\textup{dist}(x, \partial K) > (1-\ep_N)/n$, then
$$\H_d(B_d(x,(1-\ep_N)/n)\cap K)=\upsilon_d \left((1-\ep_N)/n\right)^d;$$ if $x\in \mE_{n/\ep_N}$ and $\textup{dist}(x, \partial K) \leqslant (1-\ep_N)/n$ then, for large enough $n$,
$$
\H_d(B_d(x, (1-\ep_N)/n)\cap K)\geqslant (\frac12-\delta)\upsilon_d ((1-\ep_N)/n)^d.
$$
On considering disjoint balls (relative to $K$) of radius $\ep_N/(3n)$ and using that
$$
\H_d(\{x\colon \textup{dist}(x, \partial K)\leqslant (1-\frac23\ep_N)/n\}) \leqslant C_1/n,
$$
we deduce, as in \eqref{cardbound}, that
$$
\card\left\{x\in \mE_{n/\ep_N}\colon \dist(x, \partial K)\leqslant \frac{1-\ep_N}n\right\} \leqslant C_2\frac{n^{d-1}}{\ep_N^d}.
$$
Therefore, for large enough $n$, we get
\begin{multline}\label{ballabovelong}
\P\left(\rho(X_N)> \frac1n\right) \leqslant \P\left(\exists x\in \mE_{n/\ep_N}\colon B_d\left(x, \frac{1-\ep_N}n\right)\cap K \cap X_N =\emptyset\right) \leqslant \\
C_2\frac{n^{d-1}}{\ep_N^d} \left(1- \frac{(1/2-\delta)\upsilon_d}{\H_d(K)}\left(\frac{1-\ep_N}{n}\right)^d\right)^N + C_3\frac{n^d}{\ep_N^d} \left(1- \frac{\upsilon_d}{\H_d(K)} \left(\frac{1-\ep_N}{n}\right)^d\right)^N.
\end{multline}

Repeating the estimates \eqref{problessthanlog}--\eqref{flatabovelastformula} with
$$
n_1:=\left( \frac{(1/2-\delta)\upsilon_d}{\H_d(K)} \cdot \frac{N}{\frac{d-1}d \log N + A\log\log N}\right)^{1/d},
$$
and
$$
n_2:=\left(\frac{N}{B\log N}\right)^{1/d},
$$
where $A$ and $B$ are sufficiently large, we obtain, after letting $\delta\to 0^+$, the estimate
$$
\limsup_{N\to \infty}\E[\rho(X_N)^p] \left( \frac{N}{\log N}\right)^{p/d} \leqslant \left(\frac{2(d-1)}d \cdot \frac{\H_d(K)}{\upsilon_d}\right)^{p/d}.
$$\hfill $\Box$
\subsection{Proof of the Theorem \ref{theoremball}: estimate from below}\label{sectionballbelow}
We repeat the proof from the Section \ref{sectionfrombelow}, but now we will place our net $\mE$ only on the boundary $\partial K$. Namely, put $\mE_{n/2}:=\mE_{n/2}(\partial K)$. Since $\partial K$ is a smooth $d-1$-dimensional submanifold, we see that $\card(\mE_{n/2})=\tau_n n^{d-1}$ with $0<c_1<\tau_n <c_2$. Moreover, from \eqref{fact1} we obtain as in \eqref{uniforr} that
$$
r^{-d}\H_d(B_d(x,1/n)\cap K) \rightrightarrows \frac12 \upsilon_d/n^d, \; \; r=r_n \; \mbox{or} \; r=q_n, \; n\to \infty,
$$
uniformly for $x\in \mE_{n/2}$.

The remainder of the proof just involves repeating the estimates \eqref{onemoreestimate}--\eqref{manifbelowlastformula}, using part (ii) of Lemma \ref{lemmaforf}.
\hfill $\Box$
\subsection{Estimate from above for the cube $[0,1]^d$}
The proof is similar to the case of the bodies with smooth boundary. The only change we need to make is to the formula \eqref{fact1}. Namely, if a point $x$ lies on a $(d-k)$-dimensional edge of the cube, then $\H_d(B_d(x,r)\cap [0,1]^d) \asymp 2^{-k}\upsilon_d r^d$. Moreover, $\H_d(B_d(x,r)\cap [0,1]^d) = 2^{-k}\upsilon_d r^d$ for points $x$ on the $(d-k)$-dimensional edge that are at distance larger than $r$ from all $(d-k-1)$-dimensional edges. Thus, if we consider a set $\mE_{n/\ep_N}:=\mE_{n/\ep_N}([0,1]^d)$, we have for any $k=0,\ldots, d$ at most $C_k n^{d-k}/\ep_N^d$ points $x\in \mE_{n/\ep_N}$ with $\H_d(B_d(x, (1-\ep_N)/n)\cap [0,1]^d)\geqslant 2^{-k}\upsilon_d((1-\ep_N)/n)^d$ . In particular, if $k=d$ we have only finitely many such points $x\in \mE_{n/\ep_N}$; and if $k=d-1$, we have no more than $Cn/\ep_N^d$ such points. We now repeat the estimates \eqref{problessthanlog}--\eqref{flatabovelastformula} and \eqref{ballabovelong} with
$$
n_1 := \left(2^{-(d-1)}\cdot d\cdot\upsilon_d\cdot \frac{N}{\log N+A\log\log N}\right)^{1/d}.
$$\hfill $\Box$
\subsection{Estimate from below for the cube $[0,1]^d$}
The proof is almost identical to the proof in the Section \ref{sectionballbelow}; the only difference is that now we take $\mE_{n/2}:=\mE_{n/2}(L)$, where $L$ is a $1$-dimensional edge of the cube $[0,1]^d$. To complete the analysis we appeal to part (iii) of Lemma \ref{lemmaforf}. \hfill $\Box$
\subsection{Estimates for a polyhedron in $\R^3$}
The estimates here are the same as for the unit cube $[0,1]^d$. The only difference is that, for points $x\in L$, where $L$ is the edge where two faces intersect at angle $\theta$, we have, if $x$ is far enough from the vertices of $P$:
$$
\H_3(B(x,r)\cap P) = \frac{\theta}{2\pi} \cdot \upsilon_3 \cdot r^3.
$$
Consequently, for $k=0,1,2,3$ we have at most $a_k n^{3-k}/\ep_N^3$ points $x\in \mE_{n/\ep_N}(P)$ with $\H_3(B_3(x, (1-\ep_N)/n)\cap P)\geqslant c_k\upsilon_3((1-\ep_N)/n)^3$, where $a_0=1$, $a_1=1/2$, and $a_2=\theta/(2\pi)$.
In the case $\theta\leqslant \pi/2$, one needs to choose
$$
n_1:=\left(\frac{2\theta}{V(P)}\cdot \frac{N}{\log N + A\log\log N}\right)^{1/3},
$$
and in the case $\theta\geqslant \pi/2$, one needs to choose
$$
n_1:=\left(\frac{\pi}{V(P)}\cdot \frac{N}{\log N + A\log\log N}\right)^{1/3}.
$$
For the estimate from above, consider $\mE_{n/2}(L)$ and repeat the estimates for the cube. \hfill $\Box$
\subsection{Estimates for $\textup{\textmd{d}}\mu=\frac{\textup{\textmd{d}}x}{\sqrt{1-x^2}}$}
We remind the reader that $\hat{\rho}(X_N)=\hat{\rho}(X_N, [0,1])=\sup_{y\in [1-\frac{1}{N^2}, 1]} \inf_j |y-x_j|$, where $x_j$, $j=1,\ldots,N$, are randomly and independently distributed over $[0,1]$ with respect to $\mu$.
\subsubsection{Case $a=2$}
Suppose that an interval $I_\alpha:=[1-\frac{\a}{N^2}, 1]$ is disjoint from $X_N$ for some $\alpha>1$. Then we get
$$
\hat{\rho}(X_N)\geqslant \frac{\a-1}{N^2}.
$$
We notice that if $\alpha<C_1\log^2(N)$, and $N$ is sufficiently large, then
$$
\mu(I_\alpha)\leqslant C_2 \frac{\sqrt{\a}}{N}.
$$
Therefore, if $\a$ is some number greater than $1$,
$$
\P\left(\hat{\rho}\geqslant \frac{\a-1}{N^2}\right)\geqslant \left(1-C_2\frac{\sqrt{\a}}N\right)^N \geqslant C_3.
$$
Consequently,
$$
\E\hat{\rho}\geqslant \frac{C_4}{N^2},
$$
where $C_4 = C_3(\a-1)$.

For the estimate from above, notice that $\mu(I_\a)\geqslant \sqrt{\a}/(\sqrt{2}\pi N)$. Assuming $\hat{\rho}(X_N)\geqslant \frac{\a}{N^2}$, we get that the distance from $1$ to any $x_j$ exceeds $\a/N^2$, and thus the interval $[1-\frac{\a}{N^2}, 1]$ is disjoint from $X_N$. The probability of this event is less than
$$
\left(1-C_5\frac{\sqrt{\a}}{N}\right)^N \leqslant e^{-C_5\sqrt{\a}}.
$$
Thus, for any $\a$, $1<\a<N^2$, it follows that
$$\P\left(\hat{\rho}(X_N)\geqslant \frac{\a}{N^2}\right)\leqslant e^{-C_5\sqrt{\a}}.
$$
In particular, for sufficiently large $C_6$ we have
$$\P\left(\hat{\rho}(X_N)\geqslant \frac{C_6\log^2(N)}{N^2}\right)\leqslant N^{-3}.
$$
Therefore,
$$
\E\hat{\rho}(X_N)\leqslant \frac{1}{N^2} + \sli_{\a=1}^{C_6\log^2(N)}\frac{\a+1}{N^2}e^{-C_5\sqrt{\a}} + N^{-3}.
$$
It is easy to see that the latter expression is bounded by $C_7/N^2$, which completes the proof for this case. 

\subsubsection{Case $0<a<2$}
We again notice that, if $\a$ is a number and $I=[\a, \a+\ep]\subset [1-\frac{1}{N^a},1]$ is an interval of length $\ep$, then
$$
\mu(I)=\ili_{\a}^{\a+\ep} \frac{dt}{\pi\sqrt{1-t^2}}\geqslant \frac{1}{\pi} \frac{\ep}{\sqrt{1-\a^2}}\geqslant \frac{1}{\pi} \frac{\ep}{\sqrt{1-(1-\frac{1}{N^a})^2}}\geqslant C_1 \ep N^{\frac{a}2}.
$$
Now consider $n$ intervals of length $\frac{1}{nN^a}$ (and thus having $\mu$-measure $\mu$ greater than $\frac{C_1}{nN^{\frac{a}2}}$) inside $[1-\frac{1}{N^a}, 1]$. As we have seen before, if $\hat{\rho}(X_N)>\frac{2}{nN^a}$, then for some $y\in [1-\frac{1}{N^a}, 1]$ the interval of length $\frac{2}{nN^a}$ centered at $y$ is disjoint from $X_N$; thus one of the fixed intervals of length $\frac{1}{nN^a}$ is disjoint from $X_N$. Consequently,
$$
\P\left(\hat{\rho}(X_N)\geqslant \frac2{nN^a}\right)\leqslant n\left(1-\frac{C_1}{nN^{\frac{a}{2}}}\right)^N.
$$
With
$$
n:=\frac{N^{1-\frac{a}{2}}}{A\log N},
$$
where $A$ large enough, we get
$$
\P\left(\hat{\rho}(X_N)\geqslant \frac2{nN^a}\right)\leqslant n\left(1-\frac{C_1}{nN^{\frac{a}{2}}}\right)^N \leqslant N^{-3}.
$$
Therefore,
$$
\E\hat\rho \leqslant C\frac{\log N}{N^{1+\frac{a}2}} + N^{-3},
$$
which finishes the estimate from above.

\bigskip

For the estimate from below we notice that if $I=[\a, \a+\ep]\subset [1-\frac{1}{N^a},1-\frac{1}{2N^a}]$, then
$$
\mu(I)\leqslant C_2 \frac{\ep}{\sqrt{1-(1-\frac{1}{2N^a})^2}}\leqslant C_2\frac{\ep}{N^\frac{a}{2}}.
$$

Take $n$ intervals in $[1-\frac{1}{N^a},1-\frac{1}{2N^a}]$ of length comparable to $\frac{1}{nN^a}$ and having equal $\mu$-measures $C_3\frac{1}{nN^\frac{a}2}$ (notice that if we are allowed to take such intervals near $1$, then the best measure we can get is $\frac{1}{\sqrt{nN^a}}$). If one of them is disjoint from $X_N$, then $\hat{\rho}(X_N)\geqslant \frac{C_4}{nN^a}$. Thus,
$$
\P\left(\hat\rho(X_N)\geqslant \frac{C_4}{nN^a}\right)\geqslant f\left(N, nN^{\frac{a}2}/C_3, n\right).
$$
It is easy to see that if we take
$$
n:=\frac{N^{1-\frac{a}2}}{A\log N-B\log\log N}
$$
for suitable $A$ and $B$, then the latter expression tends to one. Recall that $0<a<2$. Therefore, for large values of $N$ we have
$$
\P\left(\hat\rho(X_N) \geqslant C_4\frac{\log N}{N^{1+\frac{a}2}}\right) \geqslant \frac 12,
$$
which completes the proof for this case. 
\subsubsection{The estimate for $\tilde\rho$}
For the estimate from above simply notice that for any interval $I$ we have $\mu(I)\geqslant |I|$. For the estimate from below take the interval $[-\frac12, \frac12]$. For any interval $I\subset [-\frac12, \frac12]$ we have $\mu(I)\leqslant C|I|$, and thus the estimate from below runs as usual.

\hfill $\Box$

\end{document}